\documentclass{amsart}

\usepackage[dvipsnames,svgnames,x11names,hyperref]{xcolor}
\usepackage{bbm,url,graphicx,verbatim,amssymb,enumerate,booktabs,lmodern,mathtools,microtype,mathabx}
\usepackage[pagebackref,colorlinks,citecolor=Mahogany,linkcolor=Mahogany,urlcolor=Mahogany,filecolor=Mahogany]{hyperref}
\usepackage[capitalize]{cleveref}
\usepackage{amsmath,amscd}
\usepackage[mathscr]{euscript}
\usepackage[all]{xy}
\usepackage[margin=1.3in]{geometry}
\usepackage{tikz,tikz-cd}
\usetikzlibrary{matrix,calc,positioning,arrows,decorations.pathreplacing,patterns,arrows}

\newtheorem{theorem}{Theorem}[section]
\newtheorem*{theorem*}{Theorem}

\newtheorem{proposition}[theorem]{Proposition}

\newtheorem*{corollary*}{Corollary}

\theoremstyle{definition}
\newtheorem{definition}[theorem]{Definition}

\theoremstyle{remark}
\newtheorem{example}[theorem]{Example}

\newtheorem{remark}[theorem]{Remark}

\newcommand{\Joana}{\textcolor{teal}}

\newcommand{\map}{\textup{map}}

\newcommand{\Dec}{\mathrm{Dec}}
\newcommand{\cat}[1]{{\mathrm{#1}}}

\renewcommand{\rm}[1]{{\mathrm{#1}}}
\newcommand{\lra}{\longrightarrow}

\newcommand{\op}{\rm{op}}

\newcommand{\Fun}{\rm{Fun}}

\newcommand{\Map}{\rm{Map}}

\newcommand{\Aut}{\rm{Aut}}

\newcommand{\id}{\rm{id}}

\newcommand{\Conf}{\rm{Conf}}
\newcommand{\GT}{\rm{GT}}
\newcommand{\ZZ}{\mathbb{Z}}
\newcommand{\QQ}{\mathbb{Q}}
\newcommand{\TT}{\mathbb{T}}

\newcommand{\RR}{\mathbb{R}}

\newcommand{\CC}{\mathbb{C}}

\newcommand{\Aa}{\mathcal{A}}

\newcommand{\Dd}{\mathcal{D}}

\newcommand{\pl}{\mathrm{pl}}
\newcommand{\Weyl}{\mathfrak{W}}

\newcommand{\Ker}{\mathrm{Ker}}

\newcommand{\FCxp}[2]{\mathrm{D}^\alpha_{#1}(\mathbf{F}#2)}

\newcommand{\FCx}[2]{\mathrm{D}_{#1}(\mathbf{F}#2)}

\newcommand{\D}{\mathrm{D}}

\title{Equivariant formality of the little disks operad}
\author{Pedro Boavida, Joana Cirici and Geoffroy Horel}%

\address[P. Boavida de Brito]{Dept. of Mathematics, Instituto Superior Tecnico, Univ.~of Lisbon, Av.~Rovisco Pais, Lisboa, Portugal}%
\email{pedrobbrito@tecnico.ulisboa.pt}

\address[J. Cirici]{Departament de Matemàtiques i Informàtica, Universitat de Barcelona\\
Gran Via 585\\
08007 Barcelona, Spain}
\email{jcirici@ub.edu}

\address[G. Horel]{Université Sorbonne Paris Nord, Laboratoire Analyse, Géométrie et Applications, CNRS (UMR 7539), 93430, Villetaneuse, France.}
\email{horel@math.univ-paris13.fr}

\thanks{P. B. was supported by FCT 2021.01497.CEECIND. 
J. C. acknowledges
Govern de Catalunya (2021-SGR-00697 and Serra H\'{u}nter Program) and the Spanish State Research Agency (CEX2020-001084-M and PID2020-117971GB-C22). J.C. and G.H. are funded by the  Agence Nationale pour la Recherche through project ANR-20-CE40-0016 HighAGT and by the Centre National pour la Recherche Scientifique through project IEA00979. P.B. and G.H. are supported by a project PHC Pessoa.
}

\begin{document}

\begin{abstract} 
The little $n$-disks operad is $SO(n)$ and $O(n)$-equivariantly formal over the rationals. Equivalently, the oriented and unoriented framed little disks operads are rationally formal as $\infty$-operads.
\end{abstract}
\maketitle

\tableofcontents

The little disks operad $\mathcal{D}_n$ of dimension $n$ has a natural action of the orthogonal group $O(n)$, by rotating disks. In other words, $\mathcal{D}_n$ is an operad in $O(n)$-spaces. In this paper, we show that it is formal as such. To explain what this means precisely, recall that the equivariant homotopy type of a space $X$ equipped with an action of a group $G$ can be recorded in the Borel construction $X_{hG}$ together with its canonical map to $BG=*_{hG}$. The Borel construction defines an equivalence of $\infty$-categories from $G$-spaces (up to weak equivalence) to spaces over $BG$. In particular, applying the Borel construction aritywise to an operad $\mathcal{P}$ in $G$-spaces produces an operad in the category of spaces over $BG$. The product in the latter category is the fiber product over $BG$. Very concretely, the operadic structure is given by composition maps (over $BG$)
\[\circ_i:\mathcal{P}(n)_{hG}\times_{BG}\mathcal{P}(m)_{hG}\to\mathcal{P}(n+m-1)_{hG}\]
satisfying the usual axioms.

For a finite type nilpotent space $X$, the rational localization of $X$ can be recovered from $\Aa_{\pl}(X)$, where
$\Aa_{\pl}$ denotes Sullivan's functor of piece-wise linear forms from spaces to commutative dg-algebras.
When trying to define formality of operads from this perspective, one runs into the problem that $\Aa_{\pl}$ is not symmetric monoidal on the nose. Instead, there is a natural map
\[\Aa_{\pl}(X)\otimes \Aa_{\pl}(Y) \to\Aa_{\pl}(X\times Y)\]
which is a quasi-isomorphism. This implies that, if $\mathcal{P}$ is a topological operad, $\Aa_{\pl}(\mathcal{P})$ is not an operad (in the opposite of the category of commutative dg-algebras) on the nose but only \emph{up to coherent homotopies}. There are known ways to resolve or circumvent this. A convenient way is using the theory of $\infty$-operads through the dendroidal formalism \cite{moerdijkdendroidal}. Given a topological operad $\mathcal{P}$, one may form its dendroidal nerve which is a functor $N_d \mathcal{P}:\cat{Tree}^{\op}\to\cat{S}$ that captures all the homotopical information of the operad $\mathcal{P}$. Here $\cat{Tree}$ is the dendroidal category, also denoted $\Omega$ in the literature. We then say that an operad $\mathcal{P}$ is \emph{rationally formal} if $\Aa_{\pl}(N_d\mathcal{P})$ is formal as a diagram from $\cat{Tree}$ to commutative dg-algebras (Definition \ref{defi : formality}). This notion translates well to the equivariant setting. If our operad $\mathcal{P}$ comes with an action of a topological group $G$, we say that it is \textit{equivariantly formal} if the diagram $\Aa_{\pl}(N_d(\mathcal P_{hG}))$ from $\cat{Tree}$ to commutative dg-algebras is formal. Here $\mathcal{P}_{hG}$ is the operad in the category of spaces over $BG$ described above. In fact, its dendroidal nerve $N_d(\mathcal P_{hG})$ agrees with $(N_d\mathcal{P})_{hG}$, the homotopy orbits of the $\infty$-operad associated to $\mathcal{P}$. Having explained the terminology, we now state the main theorem.

\begin{theorem*}Let $n\geq 3$, let $G$ be $O(n)$ or $SO(n)$.
The little $n$-disks operad $\Dd_n$ is $G$-equivariantly formal.
\end{theorem*}

As a consequence, a rational model for the Borel construction of configuration spaces in $\RR^n$ with their $O(n)$ or $SO(n)$ actions is given by their cohomology rings. We compute these in Section \ref{compute}.

The homotopy orbits $(\mathcal{D}_n)_{hG}$ is an $\infty$-operad (a complete dendroidal Segal space, i.e. a fibrant object in a model category describing the homotopy theory of $\infty$-operads). This $\infty$-operad is in fact a fibrant replacement of the framed little disks (Proposition \ref{prop : orbit is framed}). So our main theorem can be rephrased as follows.

\begin{theorem*}
In dimension $n\geq 3$, the framed little disks operad and its unoriented variant are formal as $\infty$-operads.
\end{theorem*}

This constrasts with the results of Khoroshkin and Willwacher \cite{KW} and Moriya \cite{Moriya} who prove that
the framed little disks operad is not formal in odd dimensions as a one-colored operad. This is not a contradiction, but can be explained by the fact that the dendroidal nerve of the framed little disks operad is not an $\infty$-operad (with unrestricted colors): it satisfies the Segal condition but it is not complete (see also Remark \ref{rem : non contradiction}).

We also prove a relative version of the above:

\begin{theorem*}
Let $n\geq 3$ and $c\geq 2$. Let $G=O(n)$ or $G=SO(n)$. The map $\Dd_n\to\Dd_{n+c}$ induced by the standard inclusion of $\RR^n$ in $\RR^{n+c}$ is $G$-equivariantly formal.
\end{theorem*}

One motivation for such a theorem comes from embedding calculus. Embedding calculus tries to approximate spaces of embeddings between smooth manifolds by studying their effect on disks, i.e. replacing each manifold by the associated right module over the \emph{framed} little disks $\infty$-operad. Under the assumption that the manifolds are \emph{parallelized}, the framed little disks $\infty$-operad can be replaced by the standard little disks $\infty$-operad and, in this passage, the (by now classical) formality \cite{lambrechtsvolic} has been used to great effect, e.g. \cite{aronecalculus,aroneturchin}. The formality results that we prove in this paper should allow us to dispense such parallelizability assumptions and draw non-trivial consequences for the rational homotopy theory of more general embedding spaces.

Our strategy to prove these results is to use a ``purity implies formality'' technique. As input, we consider the action of the Grothendieck-Teichmüller group on the rationalized
little disks operad, built from the additivity theorem. This action is compatible with the action of the maximal torus of $O(n)$ and so induces an action on the rationalized  Borel construction. As a byproduct of our strategy we also obtain rational formality for algebraic group actions on smooth complex projective varieties, using as input the Galois action on étale cohomology. This strategy of using the Grothendieck-Teichmüller action on the two-dimensional little disks operad in order to prove its formality originated in \cite{petersenminimal}. This was then pushed to higher dimensions using the additivity theorem in \cite{boavidaformality}. The novelty of the present paper is the realization that this strategy can be pushed further to the equivariant setting.

\subsection*{Related work}
A different description of the $SO(n)$-action on the real homotopy type of $\Dd_n$ is given in \cite{KW}. The translation between our result and the main result of \cite{KW} does not seem to be straightforward as our ways of modelling equivariant homotopy type of operads are quite different. The computation of Section \ref{compute}, at least with real coefficients, can also be deduced from \cite{KW} as was explained to us by Thomas Willwacher.

\subsection*{Conventions}
All cohomology is with coefficients in $\QQ$, unless stated otherwise.
For $R$ a commutative ring, we use the notation $\cat{D}(R)$ for the $\infty$-categorical derived category of $R$ which can be defined as the localization of the category of cochain complexes over $R$ at the quasi-isomorphisms.

\section{Purity and filtered formality}\label{darkmagic}
In this section we prove abstract results on filtered complexes, generalizing the results of $E_r$-formality of filtered commutative dg-algebras of \cite{CG1} and of symmetric monoidal functors of \cite{CiHo1} in the setting of mixed Hodge structures. The main idea is that, given a filtered commutative dg-algebra $A$ together with an automorphism $\varphi$, the purity conditions on the eigenvalues of $\varphi^*$ acting on a certain page of its associated spectral sequence give a quasi-isomorphism of commutative dg-algebras between $A$ and the penultimate page of its associated spectral sequence, viewed as a commutative dg-algebra. We prove this functorially.

Throughout this section we fix a unit $\xi$ in $\QQ$ different from $\pm 1$.

\begin{definition}
Let $V$ be a $\QQ[\varphi]$-module $V$ which is finite dimensional as a $\QQ$-vector space and let $i$ be an integer. We say that $V$ is \textit{pure of weight
$i$} if either $V=0$ or $i$ is an integer and the only eigenvalue of $\varphi$ is $\xi^i$.
\end{definition}

\begin{proposition}\label{prop : Hom is zero}
Let $V$ and $V'$ be two complexes of $\QQ[\varphi]$-modules. Assume that, for all $k\geq 0$,
\begin{enumerate}
\item $H^k(V)$ and $H^k(V')$ are of finite dimension, and
\item $H^k(V)$ is pure of weight $i$ and $H^k(V')$ is pure of weight $i'$.
\end{enumerate}
Then, if $i\neq i'$, we have
$\Map_{\D(\QQ[\varphi])}(V,V')\simeq 0$.
\end{proposition}
\begin{proof}
The ring $\QQ[\varphi]$ is of homological dimension $1$ so we may assume that $V$ and $V'$ have cohomology concentrated in a single degree which can even be assumed to be zero up to shifting.
The modules $V$ and $V'$ then split as a direct sum of factors of the form $\QQ[\varphi]/(f)$ in which the roots of $f$ are $\xi^i$  and $\xi^{i'}$ respectively. The result now follows from the general fact that for coprime elements $f$ and $g$ in a
principal ideal domain $R$, we always have
$\Map_{\D(R)}(R/f,R/g)=0.$
\end{proof}

The following is an abstract statement of the fact that purity implies formality.
\begin{theorem}\label{purepurity}Let $\alpha\neq 0$ be a rational number.
When restricted to cochain complexes $A$ in $\D(\QQ[\varphi]))$ whose cohomology $H^i(A)$ is pure of weight $\alpha i$,
there is
an equivalence of strong symmetric monoidal $\infty$-functors $H\simeq \mathrm{Id}$, where
$H:\D(\QQ[\varphi])\to \D(\QQ[\varphi])$  is the cohomology functor.
\end{theorem}

Although this result can be proven directly, we will give a proof in a more general setting of filtered complexes. We start by stating some facts about filtered complexes that shall be used in what follows. We use an $\infty$-categorical language following the treatment of Antieau's paper \cite{antieauspectral}. For us a filtered complex over a base commutative ring $R$ is a diagram in the $\infty$-category $\cat{D}(R)$ indexed by the poset $(\mathbb{N},\leq )$. We denote this $\infty$-category by $\FCx{0}{R}$.

We shall often use the notation $(A,W)$ to denote a filtered complex, or just $A$ when there is no ambiguity.  The letter $A$ stands for the underlying complex, i.e. the colimit of the diagram and $W_iA$ is our notation for the $i$-th term of the filtration, so that our diagram looks like $W_0A\to W_1A\to\ldots$

We will be using the canonical filtration $\tau$ on $\cat{D}(\QQ)$. This is the filtration whose $i$-th term is given by the $-i$-connective cover (remember that we are using cohomological convention). Explicitly, for a cochain complex $C$, we have
 \[
\tau_iC^n=\left\{
\begin{array}{ll}
0&,i<n\\
\Ker(d)\cap C^n&,i=n\\
C^n&,i>n
\end{array}
\right..
\]

Note that since $\mathbb{N}$ is a directed category, one can build the projective model structure on the $1$-category $\cat{Fun}(\mathbb{N},\cat{Ch}^*(R))$ that presents the $\infty$-category $\FCx{0}{R}$. In this model structure a cofibrant object is a diagram in which each transition map is a cofibration in the standard model structure on cochain complexes. In particular, the homotopy type of any filtered complex in the $\infty$-categorical sense can be represented by a strict diagram of injective maps of complexes, that is a filtered complex in the classical sense (or a strict filtered complex in Antieau's terminology).

Given a filtered complex, one can take the homotopy cofiber of each of the transition map and obtain a collection of complexes indexed by $\mathbb{N}$ which is well-defined up to quasi-isomorphism. This defines a functor
\[\mathrm{Gr}:\FCx{0}{R}\to\cat{D}(R)^{\mathbb{N}}.\]
An easy induction shows that a map $f$ is an equivalence in $\FCx{0}{R}$ if and only if $\mathrm{Gr}(f)$ is one. Note that in our setting, filtrations are automatically complete since they are $\mathbb{N}$-indexed instead of $\mathbb{Z}$-indexed.

A filtered complex gives rise to a spectral sequence whose $E_1$-page contains the cohomology groups of the graded pieces of the filtered complex. It follows that a map of filtered complexes is an equivalence in $\FCx{0}{R}$ if and only if it induces an isomorphism on the $E_1$-page of the spectral sequence. More generally, we can define $E_r$-quasi-isomorphisms to be those maps of filtered complexes that induce isomorphisms on the $E_{r+1}$-page of the spectral sequence. Inverting such maps defines a localization of $\FCx{1}{R}$ denoted $\FCx{r}{R}$. In fact, the décalage functor relates the various localizations. In the classical setting, the
 décalage of a filtered complex $(A, W)$ is the filtered complex defined by
 \[\Dec W_i A^n:=W_{i-n}A^n\cap d^{-1}(W_{i-n-1}A^{n+1}).\]
 This is extended to the $\infty$-categorical setting in \cite[Construction 4.5]{antieauspectral}.
Note that the spectral sequences of a filtered complex and its décalage are related by a shift on their pages:
 there is a natural quasi-isomorphism of bigraded complexes
 \[E_0^{-i,j}(\Dec A)\lra E_1^{j-2i,i}(A)\]
inducing isomorphisms at the later pages (see \cite{DeHII}).

\begin{theorem}\label{theo : decalage is an equivalence}
Let $r\geq 0$. The décalage functor defines a symmetric monoidal equivalence of $\infty$-categories
\[\Dec:\FCx{r+1}{R}\to\FCx{r}{R}\]
\end{theorem}

\begin{proof}
This is essentially the content of Theorem 2.19 of \cite{CG2}. Using the discussion above, it follows that the $\infty$-category $\FCx{r}{R}$ may be obtained from the $1$-category of filtered complexes by inverting the $E_r$-quasi-isomorphisms. Finally Antieau compares his definition of décalage to Deligne's in \cite[Corollary 5.5]{antieauspectral}. The fact that $\Dec$ is symmetric monoidal follows from  \cite[Lemma 8.7]{antieauspectral}.
\end{proof}

The various pages of the spectral sequences are obviously cochain complexes. They are in fact filtered cochain complexes using the ``column filtration''
\[W_pE_r(-)^{n}:=\bigoplus_{i\leq p}E_r^{-i,n+i}(-).\]
Thus, we obtain a collection of symmetric monoidal functors
\[E_{r+i}:\FCx{r}{R}\to\FCx{0}{R},\;\textrm{for}\;i\geq 0\]
but it should be noted that, when $i>0$, these functors admit a preferred lift to the $1$-category  of filtered cochain complexes in the classical sense.

Let $\alpha\neq 0$ be a rational number. Denote $\FCxp{0}{\QQ[\varphi]}$ the full subcategory of $\FCx{0}{\QQ[\varphi]}$ spanned by those filtered complexes $(A,W)$ such that:
\begin{enumerate}[(i)]
  \item For all $i\geq 0$, $H^*(Gr_iA)$ is finite-dimensional in each cohomological degree.
  \item  If $\alpha i$ is an integer, then $E_1^{-i,j}(A)$ is pure of weight $\alpha i$. Otherwise $E_1^{-i,j}(A)=0$.
 \end{enumerate}

 For $r\geq 0$, denote by $\FCxp{r}{\QQ[\varphi]}$ the full subcategory of $\FCx{r}{\QQ[\varphi]}$
 spanned by objects $A$ such that $\Dec^r(A)$ lies in $\FCxp{0}{\QQ[\varphi]}$.
\begin{remark}\label{LisbonGate}
 A sufficient condition for a filtered complex of $\QQ[\varphi]$-modules to lie in
 $\FCxp{r}{\QQ[\varphi]}$ is that
  $E_{r+1}^{-i,j}(A)$ is finite-dimensional and pure of weight $\alpha ((1-r)i+jr)$ whenever this number is an integer, and zero otherwise. Of course this property is stable under turning pages in the spectral sequence so it suffices for it to be satisfied at some page $E_\ell$ with $\ell\leq r+1$.
\end{remark}

\begin{theorem}\label{theoEis1}
When restricted to $\FCxp{r}{\QQ[\varphi]}$, there is an equivalence of symmetric monoidal $\infty$-functors $E_r\simeq \mathrm{Id}$.
\end{theorem}

\begin{proof}
We will proceed by induction over $r\geq 0$.

The base case of the induction is $r=0$. In this case, consider the two functors
\[\mathrm{Gr}:\FCxp{0}{\QQ[\varphi]}\leftrightarrows \mathrm{D}(\QQ[\varphi])^{\mathbb{N}}:U\]
where $\mathrm{Gr}$ is the associated graded functor and $U$ is the functor that takes a graded complex $\{A_k\}_{k\in\mathbb{N}}$ to the filtered complex $(UA,W)$ given by
\[W_iU(A):=\bigoplus_{k\leq i}A_k\]
with obvious transition maps. We claim that the functor $\mathrm{Gr}$ is fully faithful. Assuming this for the moment, the result can be deduced as follows. In order to prove that $E_0$ is equivalent to $\mathrm{Id}$, it suffices to prove that $\mathrm{Gr}\circ E_0$ is equivalent to $\mathrm{Gr}$. The functor $E_0$ is simply the composite $U\circ \mathrm{Gr}$, since $\mathrm{Gr}\circ U\cong\mathrm{Id}$, the result follows immediatly.

Now, we prove that $\mathrm{Gr}$ is fully faithful. Let $C$ and $D$ be two objects in $\FCxp{0}{\QQ[\varphi]}$ whose filtrations are both denoted by $W$. We wish to prove that the map
\[\Map_{\FCxp{0}{\QQ[\varphi]}}(C,D)\to\prod_i \Map_{\D(\QQ[\varphi])}(\mathrm{Gr}_i(C),\mathrm{Gr}_i(D))\]
is an equivalence.

We view both the source and target as functors of $C$ and $D$. Clearly, both the source and target preserve colimits in the $C$ variable. We have an equivalence
\[C\simeq \mathrm{colim}_i W_iC\]
where $W_iC$ is equipped with the induced filtration~:
\[W_0C\to\ldots\to W_{i-1}C\to W_iC\xrightarrow{\id}W_iC\xrightarrow{\id}\ldots\]
It follows that we can reduce the proof to the case where $C$ has a finite filtration. Now, we observe that we have a cofiber sequence
\[W_{i-1}C\to W_iC\to \mathrm{Gr}_iC\]
which can be viewed as a cofiber sequence of filtered objects in which $Gr_iC$ is equipped with the filtration
\begin{align*}
W_j (\mathrm{Gr}_iC)&=\mathrm{Gr}_iC\;\textrm{if}\;i\leq j\\
         &=0\;\textrm{else.}
\end{align*}
From this observation, we see that we may reduce to the case in which $C$ is
such that $W_pC=0$ for $p<i$ and $W_pC=\mathrm{Gr}_i(C)$ for $p\geq i$.
For this specific $C$, we see that the map that we are considering is simply
\[\Map_{\D(\QQ[\varphi])}(\mathrm{Gr}_iC,W_iD)\to \Map_{\D(\QQ[\varphi])}(\mathrm{Gr}_iC,\mathrm{Gr}_iD)\]
In order to prove that this map is an equivalence, it is enough to prove that its fiber is zero. This fiber is simply the mapping space $\Map_{\D(\QQ[\varphi])}(\mathrm{Gr}_iC,W_{i-1}D)$. By a similar inductive argument, we may reduce to proving that
\[\Map_{\D(\QQ[\varphi])}(\mathrm{Gr}_iC,\mathrm{Gr}_{k}D)=0\]
for all $k<i$. This follows from the proposition \ref{prop : Hom is zero} and finishes the proof in the case $r=0$.

Now take $r>0$ and assume the result has been proved for $r-1$, then we  use Theorem \ref{theo : decalage is an equivalence}. We obtain an equivalence
\[E_r\simeq \Dec^{-1}\circ E_{r-1}\circ \Dec \simeq \Dec^{-1}\circ \mathrm{Id}\circ \Dec\simeq \mathrm{Id}. \qedhere\]
\end{proof}

\begin{proof}[Proof of Theorem \ref{purepurity}]
Given a complex $A$, endow it with the canonical filtration $\tau$. Then, the purity condition on $H^*(A)$
makes $(A,\tau)$ into an object in $\FCxp{0}{\QQ[\varphi]}$. Therefore $(E_0(A),d_0)$ is naturally quasi-isomorphic to $A$. But since $\tau$ is the canonical filtration, we have $E_0(A)\simeq H(A)$.
\end{proof}

\section{Equivariant homotopy theory}

\subsection{Borel construction and rational equivariant homotopy theory}
For a topological group $G$ acting on a space $X$, one may form the Borel construction $X_{hG}$. This is a space with a map to $BG$. Explicitly, this can be constructed as
\[X_{hG}:=(X\times EG)/G\to (*\times EG)/G:=BG\]
where $EG$ denotes any contractible CW complex with free $G$-action and the quotient is taken with respect to the diagonal $G$-action. It is classical that the homotopy type of $X$ with its $G$ action can be reconstructed from this map. Indeed, we have a fiber sequence
\[X\to X_{hG}\to BG\]
and the monodromy action of $G\simeq \Omega BG$ coincides with the $G$-action on $X$. This can be expressed by the following theorem.

\begin{theorem}
The Borel construction defines an equivalence of $\infty$-categories
\[\Fun(BG,\cat{S})\to \cat{S}/BG,\]
where $\cat{S}$ denotes the $\infty$-category of spaces.
\end{theorem}

Given a $G$-space $X$, it is customary to denote by $H^*_G(X)$ the cohomology of the Borel construction. This is called the \textit{equivariant cohomology} of $X$. Likewise,
the \textit{rational equivariant homotopy type} is defined to be the homotopy type of the map of commutative dg-algebras
\[\Aa_{\pl}(BG)\to \Aa_{\pl}(X_{hG}).\]

The following definition is fundamental in this paper. We denote by $\cat{CDGA}$ the $1$-category of commutative differential graded algebras. Observe that a commutative graded algebra can be viewed as an object of $\cat{CDGA}$ by giving it the trivial differential.

\begin{definition}\label{defi : formality}
Let $I$ be a small category. A diagram $I\to\cat{CDGA}$ is called \emph{formal} if there is a zig-zag of natural transformations connecting this diagram to its cohomology in which each map is objectwise a quasi-isomorphism.
\end{definition}

\begin{remark}
There are several possible variants of this definition that are more or less strict. They are in fact all equivalent thanks to the following chain of equivalences of $\infty$-categories:
\[\cat{Fun}(I,\cat{CDGA}_{\QQ})[\mathcal{W}^{-1}]\simeq \cat{Fun}(I,\cat{CDGA}_{\QQ}[\mathcal{W}^{-1}])\simeq \cat{Fun}(I,\cat{CAlg}(\cat{D}(\QQ)))\]
The first of these is the localization of the $1$-category of functors at the objectwise quasi-isomorphisms, the second one is the $\infty$-category of functors from $I$ to the localization of $\cat{CDGA}$ at the quasi-isomorphisms, finally the last one is the $\infty$-category of functors from $I$ to the $\infty$-category of commutative algebras in $\cat{D}(\QQ)$.  The first equivalence follows from \cite[Proposition 4.2.4.4]{HTT} and the second one follows from \cite[Theorem 4.1.1]{hinichrectification}. Then, a diagram is formal in the sense of Definition \ref{defi : formality} if and only if it is equivalent to its cohomology when both are viewed as objects of the leftmost $\infty$-category, but it is in fact enough to prove that they are equivalent as objects of the rightmost category.
\end{remark}

As particular case of the above definition we have the following.

\begin{definition}
We say that a $G$-space $X$ is \textit{equivariantly formal} if the map
\[\Aa_{\pl}(BG)\to \Aa_{\pl}(X_{hG})\]
is formal when we view it as a functor from the poset $[1]=\{0<1\}$ to $\cat{CDGA}$.
\end{definition}

\begin{remark}
The above notion of equivariant formality implies the multiplicative collapse of  the Eilenberg-Moore spectral sequence
\[\mathrm{Tor}^{i,j}_{H^*(BG)}(H^*_G(X),\QQ)\Rightarrow H^{i+j}(X).\]
\end{remark}

\begin{remark}
Formality of the Borel construction of a $G$-space $X$ does not necessarily imply formality of the action map $G\times X\to X$. For instance, the Borel construction of $S^3=SU(2)$ acting on $S^2=\mathbb{CP}^1$ is formal. On the other hand, the action map
\[\alpha:S^3\times S^2\to S^2\]
is not formal. Indeed, any formal map from the rationalization of $S^3\times S^2$ to the rationalization of $S^2$ must factor through the projection $S^3\times S^2\to S^2$. On the other hand, for any point $*\in S^2$, the composite
\[S^3=S^3\times *\to S^3\times S^2\xrightarrow{\alpha} S^2\]
coincides with the Hopf map and in particular is not rationally nullhomotopic. Another example is the $S^1$ action on $S^2$ by rotation, the Borel fibration is formal (c.f. section \ref{eq:2pts}) but the action map is not formal (otherwise the Borel fibration would be rationally trivial).
\end{remark}

\begin{remark}
The terminology \emph{equivariantly formal} is often used in the literature for something quite unrelated, namely the condition that the cohomology of the Borel construction $H^*(X_{hG})$ is free as a $H^*(BG)$-module. This is equivalent to asking that the Leray-Serre spectral sequence of the Borel fibration collapses at the $E_2$-page, and implies that the $E_2$-page of the Eilenberg-Moore spectral sequence is concentrated in the column $i = 0$. Neither notion of equivariant formality implies the other. We believe that our terminology is more compatible with the usage in rational homotopy theory. For us a $G$-space is equivariantly formal if its rational homotopy type is a formal consequence of the equivariant cohomology algebra $H^*_G(X)$ as an algebra over $H^*(BG)$. This notion is also called \textit{$G$-formality} in \cite{Lilly} and \textit{bundle formality} in \cite{Sculltorus},
where the various notions of equivariant formality are compared.
\end{remark}

\subsection{Leray-Serre spectral sequence for the Borel construction}

We next recall the construction of the Leray-Serre spectral sequence from the $\infty$-categorical point of view. Given a map of spaces $X\to Y$, there is an adjunction of $\infty$-categories
\[f^*:\mathrm{Fun}(Y,\D(\QQ))\leftrightarrows \mathrm{Fun}(X,\D(\QQ)):f_*\]
where the functor $f^*$ is precomposition with $f$ and $f_*$ is its right adjoint. If $p:X\to *$ is the unique map, then $p_*$ is a lax monoidal functor and $p_*(\QQ)$ is thus a commutative algebra object in the $\infty$-category $\D(\QQ)$ which is quasi-isomorphic to the singular cochains $E_\infty$-algebra $C^*(X,\QQ)$ and also to the commutative algebra $\Aa_{\pl}(X)$.

\begin{remark}
If $X$ is a reasonable space (paracompact, Hausdorff and  locally contractible) the category $\Fun(X,\D(\QQ))$ can be modelled by the full subcategory of the $\infty$-category of $\D(\QQ)$-valued sheaves on $X$ spanned by the locally constant sheaves. The functor $p_*$ coincides with what is usually called $\mathrm{R} p_*$ in the sheaf literature. In this case the multiplicative structure on $\mathrm{R}p_*\QQ$ can be strictified to a strict commutative dg-algebra using the Thom-Whitney construction \cite{Na}.
\end{remark}

For a general map $f:X\to Y$, we denote by $p:X\to *$ and $q:Y\to *$ the unique maps and we construct a filtration of $p_*(\QQ)=q_*(f_*(\QQ))$ by $q_*(f_*\QQ,\tau)$, namely, the image of the canonical filtration $\tau$ under $q_*$. The canonical filtration on $\cat{Fun}(Y,\cat{D}(\QQ))$ is obtained by applying the canonical filtration pointwise.

The resulting spectral sequence is given by
\[E_1^{-i,j}=H^{j-i}(X,H^i (f_*\QQ)[-i])=H^{j-2i}(Y,\mathcal{H}^i(F))\]
where $F$ denotes the fiber of $f$ and $\mathcal{H}^i(F)$ denotes the local system $H^i(F)$ equipped with its monodromy action by the fundamental groupoid of $Y$.
The above $E_1$-page is isomorphic to the page ${}^{\mathrm{LS}}E_2^{j-2i,i}$ of the Leray-Serre spectral sequence. In fact the two spectral sequences coincide up to renumbering (see Example 1.4.8 of \cite{DeHII} or \cite[Proposition 9.2]{antieauspectral})

Now we specialize to the case of a $G$-space $X$, where $G$ is a compact Lie group.
We denote by $\Aa_G(X)$ the filtered complex $q_*(f_*\QQ,\tau)$ producing the Leray-Serre spectral sequence for the map $f:X_{hG}\to BG=*_{hG}$, where $q:BG\to *$.
The associated spectral sequence of $\Aa_G(X)$ gives
\[E_1^{-i,j}(\Aa_G(X))\cong {}^{\mathrm{LS}}E_2^{j-2i,i}(X_{hG})=H^{j-2i}(BG,H^i(X))\Rightarrow H^{j-i}_G(X).\]

Recall from Section \ref{darkmagic} that $\FCx{0}{\QQ}$ is the filtered derived category, defined by localizing at $E_1$-isomorphisms. We have:

\begin{proposition}\label{prop functoriality}
The assignment
\[X\mapsto \Aa_G(X)\]
defines a contravariant symmetric monoidal $\infty$-functor from $G$-spaces to $\FCx{0}{\QQ}$. For every $r\geq 1$, the assignment
\[X\mapsto (E_r(\Aa_G(X)),d_r)\]
is a contravariant symmetric monoidal $\infty$-functor from $G$-spaces to the $1$-category $\cat{CDGA}$.
\end{proposition}

\begin{proof}
By the discussion above, this functor is identified with the composite of
\[\cat{Fun}(BG,\cat{S})\simeq\cat{S}_{/BG} \to\cat{Fun}(BG,\cat{D}(\QQ))^\op\to\cat{Fun}(BG,\FCx{0}{\QQ})^\op\xrightarrow{q_*}\FCx{0}{\QQ}^\op.\]
where the first functor is induced by postcomposition with the singular cochain functor $\cat{S}\to\cat{D}(\QQ)$ and the second functor is induced by postcomposition with the canonical filtration functor
\[\cat{D}(\QQ)\to\FCx{0}{\QQ}.\]
The first two functors of this sequence are symmetric monoidal and the last one is oplax symmetric monoidal. In particular, since any $G$-space is a commutative coalgebra, the object $\Aa_G(X)$ is naturally a commutative algebra in $\FCx{0}{\QQ}$. The second claim follows from the fact that $E_r$ is a lax symmetric monoidal functor to the $1$-category of cochain complexes (see \cite[Theorem 8.9]{antieauspectral}).
\end{proof}

\subsection{The orthogonal group}\label{subsection : orthogonal stuff}

We recall some very classical facts about the orthogonal group $O(n)$ \cite{Borel}. This is a compact Lie group. If $n$ is even, $n=2m$, then a maximal torus of $O(n)$ is given by the $2m\times 2m$ matrices that have $m$ blocks of size $2\times 2$ on the diagonal that are elements of $SO(2)$. If $n$ is odd, $n=2m+1$, a maximal torus of $O(n)$ is given  by the $(2m+1)\times (2m+1)$ matrices that have $m$ blocks of size $2\times 2$ on the diagonal that are elements of $SO(2)$ and where the remaining diagonal element is a $1$.

In the even case $n=2m$ the Weyl group is given by the semi-direct product $\Sigma_m\ltimes C_2^m$. This can be realized as a subgroup of $O(2m)\subset \mathrm{Aut}_{\RR}(\CC^m)$. The subgroup $\Sigma_m$ acts on $\CC^m$ by permuting the factors while the subgroup $C_2^m$ acts on $\CC^m$ by applying complex conjugation to some of the factors.

In the odd case $n=2m+1$, the Weyl group is given by $(\Sigma_m\ltimes C_2^m)\times C_2$. This can be realized as a subgroup of $O(2m+1)\subset \mathrm{Aut}_{\RR}(\CC^m\oplus\RR)$ in which $(\Sigma_m\ltimes C_2^m)$ acts on the $\CC^m$ factor as in the previous case and the extra copy of $C_2$ acts by the sign representation on $\RR$ and fixes $\CC^m$.

If instead of $O(n)$ we want to work with $SO(n)$, the maximal torus is the same as for $O(n)$ and the Weyl group is a subgroup of the Weyl group of $O(n)$ of index $2$. In the even case $n=2m$ it is the subgroup of $\Sigma_m\ltimes C_2^m$ consisting of elements $(\sigma,\epsilon_1,\ldots,\epsilon_m)$ such that
\[\mathrm{sign}(\sigma) \epsilon_1 \cdots \epsilon_m =1.\]
In the odd case $n=2m+1$, it is the subgroup $(\Sigma_m\ltimes C_2^m)\times C_2$ consisting of elements $((\sigma,\epsilon_1,\ldots,\epsilon_m),\eta)$ such that
\[\mathrm{sign}(\sigma) \epsilon_1 \cdots \epsilon_m \eta=1.\]

\subsection{Rational equivariant homotopy theory for the orthogonal group}

Let $G$ be a Lie group. Denote $\TT$ its maximal torus and $\Weyl$ its Weyl group.
Given a $G$-space $X$, we may form the Borel construction $X_{h\TT}$ for its $\TT$-action. This comes with a map to $B\TT$ and this map is $\Weyl$-equivariant. The following proposition (see \cite{atiyahbott} for example) shows that the equivariant rational homotopy type for actions of compact Lie groups can be reduced to the case of tori.

\begin{proposition}
Let $X$ be a $G$-space. The map
\[\Aa^*_{\pl}(X_{hG})\to\Aa^*_{\pl}(BG)\]
is naturally quasi-isomorphic to the map obtained by applying $\Weyl$ fixed points to
\[\Aa^*_{\pl}(X_{h\TT})\to\Aa^*_{\pl}(B\TT)\]
\end{proposition}

\begin{remark}
Since $\Weyl$ is a finite group and we are working rationally, the fixed point functor is exact and therefore does not need to be derived.
\end{remark}

From the previous subsection and the above proposition we deduce that there is a natural quasi-isomorphism
\[\Aa^*_{\pl}(X_{hO(n)})\simeq (\Aa^*_{\pl}(X_{hSO(n)}))^{C_2}\]
where $C_2$ is the quotient of the Weyl group of $O(n)$ by the Weyl group of $SO(n)$.

As an application of this proposition let us recall the computation of the cohomology of $BO(n)$ and $BSO(n)$.

\begin{proposition}
Let $n\geq 2$ be an integer. Then
\begin{itemize}
\item If $n$ is odd
\[H^*(BSO(n))\cong H^*(BO(n))\cong\QQ[p_1,p_2,\ldots,p_{\lfloor n/2\rfloor}]\]
with $p_i$ of degree $4i$.
\item If $n$ is even $H^*(BO(n))\cong\QQ[p_1,\ldots,p_{n/2}]$ with $p_i$ of degree $4i$.
\item If $n$ is even $H^*(BSO(n))\cong\QQ[p_1,\ldots,p_{n/2-1},e]$ with $p_i$ of degree $4i$ and $e$ of degree $n$.
\end{itemize}
Moreover, for $n$ even the map $H^*(BO(n))\to H^*(BSO(n))$ sends $p_{n/2}$ to $e^2$ (and the other generators to the generators of the same name).
\end{proposition}

\section{Homotopy orbits of the little disks $\infty$-operad}

Let $\Dd_n$ denote the little $n$-dimensional disks operad. There is an action of $O(n)$ on $\Dd_n$ by rotating disks. That is, each space $\Dd_n(k)$ has an $O(n)$-action and the composition maps
\[
\Dd_n(k) \times \Dd_n(m_1) \times \dots \times \Dd_n(m_k) \to \Dd_n(m_1 + \dots + m_k)
\]
are $O(n)$-equivariant, where the source is given the diagonal action. In order to make sense of the \emph{homotopy orbits} of $\Dd_n$ with respect to $O(n)$ it is a good idea to treat the space of colours of the operad in the same way as the space of operations. One way to do this, which is also robust in homotopical terms, is to take the dendroidal nerve of $\Dd_n$. This is the functor from category $\cat{Tree}$ (denoted $\Omega$ in  \cite{moerdijkdendroidal}) to spaces which to a tree $T$ assigns the product, over the vertices $v$ of $T$ with $|v|$ incoming edges, of $\Dd_n(|v|)$. For the tree $\eta$, with no vertices, we assign the point. We will make no notational distinction between the operad and its dendroidal nerve.

We denote by $f\Dd_n$ the framed little disks operad, it is a topological operad given by
\[f\Dd_n(k)=\mathrm{Emb}(D^{\sqcup k},D)\]
where $D$ denotes a standard open $n$-disk and $\mathrm{Emb}$ denotes the space of smooth embeddings. In fact this operad comes in two flavors depending of whether we consider orientation preserving embeddings or arbitrary embeddings. We use the notation $f\Dd_n^+$ for the orientation preserving version and $f\Dd_n$ for the other one.

For $G=O(n)$ or $SO(n)$, we define the homotopy $G$-orbits of $\Dd_n$ to be the dendroidal space $(\Dd_n)_{hG}$ obtained by taking homotopy orbits objectwise. It is an $\infty$-operad, in the following sense.

\begin{proposition}For $G=O(n)$ or $SO(n)$,
the dendroidal space $(\Dd_n)_{hG}$ satisfies the Segal and completeness conditions.
\end{proposition}
\begin{proof}
The proof will work for any $1$-reduced dendroidal Segal space $X$ acted on by a group $G$ so we prove it in this generality (for us $1$-reduced means that the value on $\eta$ and one the $1$-corolla is contactible).
Let $T$ be a tree given as the union of two trees $R$ and $S$ intersecting along an edge $e$. To verify the Segal condition, we must check that the map
\[
X_{hG}(T) \to X_{hG}(R) \times^h_{X_{hG}(e)} X_{hG}(S)
\]
is a weak equivalence. But this is clear since this map is fibered over $BG$ and the fiber is the isomorphism $X(T) \to X(R) \times X(S)$.
Completeness asserts that the restriction of $X_{hG}$ to linear trees is a complete Segal space in the sense of Rezk, i.e., that the inclusion of the space of objects in the space of homotopy invertible morphisms is an equivalence \cite[Section 6]{rezkmodel}. This is also immediate since that restriction is the homotopically constant simplicial space $BG$.
\end{proof}

The following proposition shows that the homotopy orbits of $\mathcal{D}_n$ is the $\infty$-operad associated to the framed little $n$-disks operad. This is not new and, for a different model for $\infty$-operads, follows from results of \cite{lurie} as explained in \cite[Proposition 2.2]{horel2022two}.

\begin{proposition}\label{prop : orbit is framed}
The Rezk completion of $f\Dd_n$ is weakly equivalent to the $O(n)$ homotopy orbits of $\Dd_n$. Likewise, the Rezk completion of $f\Dd_n^+$ is weakly equivalent to the $SO(n)$ homotopy orbits of $\Dd_n$.
\end{proposition}
\begin{proof}
We write the proof in the $O(n)$ case. The other case is completely similar. In preparation for the proof, we will need to introduce some notation. Given a tree $T$, we write $e(T)$ for its set of edges and $\lambda(T) \subset e(T)$ for its subset of leaves (the root does not qualify as a leaf). The group $\map(e(T), O(n))$ will be denoted $O(T)$.

We will write $X$ as short for $f\Dd_n$. For a corolla $C$, the space $X(C)$ has a natural action of $O(C)$ by pre and post-composition.  Namely, given a smooth embedding $f : \lambda(C) \times \RR^d \to \RR^d$ and an element $g:=(g_e) \in O(C)$, the element $g \cdot f$ is the embedding $\lambda(C) \times \RR^d \to \RR^d$ whose restriction to the component corresponding to $i \in \lambda(C)$ is $g_{e_0} f(i, -) g_{i}^{-1}$, where $e_0$ denotes the root. Correspondingly, for a general tree $T$ the space $X(T) = \prod_{v} X(C_v)$ has an action of $O(T)$. (Note $e(T)$ is the union of $e(C_v)$ where $v$ runs over the vertices of $T$.) These actions are compatible with respect to morphisms in $\cat{Tree}$. Given a morphism $R \to T$ in $\cat{Tree}$, we have a map $e(R) \to e(T)$ and a corresponding map $O(T) \to O(R)$. This endows $X(R)$ with an action of $O(T)$. Then, the induced map
\[
X(T) \to X(R)
\]
is $O(T)$-equivariant. So we let $Z$ be the dendroidal space whose value at a tree $T$ is the space of homotopy orbits
\[
Z(T) := X(T)_{h O(T)} \; .
\]
\emph{Claim}: $Z$ is a complete Segal dendroidal space and the map $X \to Z$ is a Rezk completion.

Verifying the Segal condition is similar to the proposition above. Let $T$ be a tree decomposed as a union of two trees $R$ and $S$, intersecting along an edge $u$. We need to show that the canonical map
\begin{equation}\label{eq:segal1}
X(T)_{hO(T)} \to X(R)_{hO(R)} \times^h_{X(u)_{hO(u)}} X(S)_{hO(S)}
\end{equation}
is a weak equivalence. This map fibers over
\begin{equation}\label{eq:segal2}
BO(T) \to BO(R) \times^h_{BO(u)} BO(S)
\end{equation}
where $BO(T) \cong \map(e(T), BO(n))$ and similarly for the other terms. In particular, the map \eqref{eq:segal2} is a weak equivalence, as $e(T) = e(R) \cup_{e(u)} e(S)$. The fiber of the map \eqref{eq:segal1} over the map \eqref{eq:segal2}
\[
X(T) \to X(R) \times^h_{X(e)} X(S)
\]
which is a weak equivalence since $X$ is Segal. Therefore \eqref{eq:segal1} is a weak equivalence, and so $Z$ is Segal. It is also complete, since the restriction of $Z$ to linear trees is weakly equivalent to the constant simplicial space $BO(n)$. The canonical map $X \to Z$ is, when restricted to linear trees, the completion map from the nerve of $O(n)$ to the constant simplicial space $BO(n)$. In particular, it is trivially essentially surjective. So to show that $X \to Z$ is a Rezk completion, it remains to show that it is fully faithful. That means showing, for every corolla $C$ with $n$ edges, that the square
\[
\xymatrix@M=5pt@R=18pt{
X(C) \ar[r] \ar[d] & \ar[d] Z(C) \\
\prod_{e \in C} X(e) \ar[r] & \prod_{e \in C} Z(e)
}
\]
is homotopy cartesian. But this is tautologically true, since $X$ evaluated on a tree with no vertices is a point, and the remaining terms form the homotopy fiber sequence $X(C) \to X(C)_{hO(C)} \to BO(C)$.

To conclude the proof, observe that the map $\Dd_n \to f\Dd_n \to Z$ factors through $(\Dd_n)_{hO(n)}$ and the resulting factorization $(\Dd_n)_{hO(n)} \to Z$ is a degreewise weak equivalence,  completing the proof.
\end{proof}

\begin{remark}\label{rem : non contradiction}
The model category of complete dendroidal Segal spaces presents the $\infty$-category of $\infty$-operads (with no restrictions on colors). Taking a fibrant replacement in this model structure can change a dendroidal space quite significantly. For example, take $G$ a topological group and view it as a topological operad concentrated in arity $1$. Then $N_d(G)$ is the dendroidal space that sends the linear tree $[n]$ to $G^n$ and any other tree to $\varnothing$. This is not complete unless $G$ is contractible, and a fibrant replacement is the dendroidal space that sends any linear tree to $BG$ (and non-linear trees to $\varnothing$). In particular, formality of these two diagrams are very different conditions. In the first case, it means formality of the Hopf algebra $\Aa_{\pl}(G)$ whereas in the second case it means formality of $\Aa_{\pl}(BG)$.

It should be noted that being one-colored is not a property of an $\infty$-operad, it is additional structure. This can be made precise by observing that the nerve of a one-colored simplicial operad $\mathcal{P}$ is usually not fibrant in the model category of complete dendroidal Segal spaces, i.e., it does not represent an $\infty$-operad. The Segal condition will always hold but the completeness can fail (it holds if and only if the space of homotopy units of the monoid $\mathcal{P}(1)$ is contractible). On the other hand, the nerve of $\mathcal{P}$ does represent a \emph{one-colored} $\infty$-operad, e.g., it is fibrant in a model category of dendroidal spaces equivalent to that of one-colored operads in spaces (the construction of such a model category can be found in \cite{bergnergroup}). This explains the apparent discrepancy between our main result and the results of \cite{Moriya} and \cite{KW}. In short, the framed little disks operad is not formal  in odd dimensions as a one-colored operad or one-colored $\infty$-operad but, as we show, it is always formal as an $\infty$-operad.
\end{remark}

\section{Little disks and some special automorphisms}

\subsection{A special automorphism of $\Dd_2$}

Recall that the group $\textup{GT}_{\QQ}$ consists of pairs $(\lambda,f)$ where $\lambda\in \QQ^\times$ and $f$ is an element of the Malcev completion of $F_2$. These pairs are required to satisfy some equations (see \cite[Equation 4.3,4.4 and 4.10]{drinfeldGT}). In particular, the pair $\tau=(-1,1)$ is an element of $\textup{GT}_{\QQ}$. The map
\[(\lambda,f)\mapsto \lambda\]
defines a group homomorphism $\chi:\textup{GT}_{\QQ}\to\QQ^\times$ called the \emph{cyclotomic character}.

\begin{proposition}\label{prop : dimension2}
Let $\psi \in \QQ^{\times}$. There exists an element $\varphi \in \textup{GT}_{\QQ}$ whose cyclotomic character is $\psi$ and which commutes with $\tau$.
\end{proposition}

\begin{proof}
To prove this proposition, we recall some facts about the theory of associators from the operadic point of view (see \cite{Fresse}, \cite{BaNa}, see also \cite{calaqueroca} for a survey). There is an operad $PaB$ in the category of groupoids whose classifying space is a model for the little $2$-disks operad. Applying Malcev completion aritywise yields an operad $PaB_{\QQ}$ in the category of groupoids enriched in commutative coalgebras. There is another operad $PaCD_{\QQ}$ in the category of groupoids enriched in commutative coalgebras which is the associated graded of $PaB_{\QQ}$ for the filtration by powers of the augmentation ideal. We can extend the scalars to any commutative $\QQ$-algebra $R$ and produce $PaB_R$ and $PaCD_R$. From this perspective  an associator over $R$ is an isomorphism of operads
\[PaB_R\to PaCD_R\]
which is the identity on the operad of objects (recall that these are operads in groupoids enriched in coalgebras). The arity $2$-component of an associator is determined by a map of Hopf algebras
\[
R[t,t^{-1}]^{\wedge}_I\to R[[t]]
\]
where the source denotes completion with respect to the augmentation ideal.
Such a map is necessarily of the form $t\mapsto \exp(\lambda t)$ for some $\lambda$ in $R$. A $1$-associator is an associator for which $\lambda=1$. This can be very explicitly described as a power series $\Phi\in R\langle\langle X,Y\rangle\rangle$ which is group-like and satisfies a pentagon and hexagon identity. This description, in turn, comes from a presentation of the operad $PaB$ (see \cite[Theorem 2.30]{calaqueroca}).

Similarly, we may introduce the group $\textup{GT}_R$ and $\textup{GRT}_R$ of automorphisms of $PaB_R$ and $PaCD_R$, respectively, inducing the identity on objects. These two groups are functorial in $R$ and can be viewed as affine group schemes. Evaluation in arity $2$ induces morphisms
\[\textup{GT}_R\to R^\times\;\textrm{and}\;\textup{GRT}_R\to R^\times\]
that we call the cyclotomic characters. Conjugation with an associator over $R$ induces an isomorphism of groups $\textup{GT}_R\cong \textup{GRT}_R$ which further commutes with the two cyclotomic characters.

Finally the operads $PaB_R$ and $PaCD_R$ both admit an involution corresponding topologically to the complex conjugation action on the little $2$-disks operad. An even associator over $R$ is defined to be an isomorphism $PaB_R\to PaCD_R$ which is a $1$-associator and further commutes with the involutions. This can be explicitly described as a power series $\Phi$ as above which further satisfies the equation
\begin{equation}\label{involution}
\Phi(-X,-Y)=\Phi(X,Y)
\end{equation}
This fact can be seen as a consequence of the presentation of $PaB$. An associator is determined by where it sends the arity three generator $\mu$ of $PaB$ (denoted $\Phi$ in \cite[Theorem 2.30]{calaqueroca}). This generator $\mu$ is invariant under the involution, so an associator is even if and only if it sends $\mu$ to an element of $PaCD_R$ which is invariant under the involution. If we translate this condition into the power series $\Phi$ as in \cite[Corollary 2.32]{calaqueroca}, we see that the condition is exactly equation (\ref{involution}).

Finally, the main point is that Drinfel'd has shown \cite[Proposition 5.4]{drinfeldGT} that rational even associators exist. Moreover, the group $\textup{GRT}_{\QQ}$ splits as a semi-direct product $\textup{GRT}_{\QQ}\cong \textup{GRT}^1_{\QQ}\rtimes\QQ^\times$. It follows that we can obtain the desired element of $\textup{GT}_{\QQ}$ by conjugating $(1,\psi)\in \textup{GRT}_{\QQ}$ with any rational even associator.
\end{proof}

\subsection{A special automorphism of $\Dd_n$}

For $X$ a space, we denote by $X_{\QQ}$ the rationalization of $X$ in the sense of Bousfield. The functor $X\mapsto X_{\QQ}$ is symmetric monoidal in the $\infty$-categorical sense. In particular, if $\mathcal{P}$ is a topological operad that we think of as a dendroidal space satisfying Segal condition, we may form $\mathcal{P}_{\QQ}$ by applying rationalization objectwise. Since the functor $\Aa_{\pl}$ sends rational equivalences to quasi-isomorphisms, we have an equivalence of functors from $\cat{Tree}$ to $\cat{CDGA}$
\[\Aa_{\pl}(\mathcal{P})\simeq \Aa_{\pl}(\mathcal{P}_{\QQ}).\]
This is true both in the strict sense and in the $\infty$-categorical sense.

\begin{proposition}\label{automorphism}
Let $\mu \in \QQ^{\times}$, let $n\geq 3$. There is an automorphism $\varphi$ of $(\Dd_{n})_{\QQ}$ such that
\begin{enumerate}
\item The induced action of $\varphi^*$ on $H^{i}(\Dd_n)$ is by multiplication by $\mu^{i/2}$ if $n$ is odd and by $\mu^{ni/2(n-1)}$ if $n$ is even.
\item The automorphism $\varphi$
can be extended to an automorphism of the $\infty$-operad $((\Dd_n)_{\QQ})_{hO(n)}$.
\item The induced action of $\varphi^*$ on $H^{i}(BO(n))$ is by multiplication by $\mu^{i/2}$.
\end{enumerate}
\end{proposition}

Before proving this proposition, let us make some comments.
Note that the cohomology of $\Dd_n$ is concentrated in degrees that are multiples of $n-1$, and the cohomology of $BO(n)$ is concentrated in even degrees, so all the exponents of $\mu$ that appear in the proposition are integers.

 Condition (2) implies in particular that $BO(n)_{\QQ}$ has an action of $\varphi$ since this space is the value of $((\Dd_n)_{\QQ})_{hO(n)}$ at linear trees. This is the action that we refer to in part (3). Moreover (2) is asserting that the action of $\varphi$ on $\Dd_{2n}$ is $O(n)$-equivariant in a ``twisted'' sense. Precisely, it does not commute with the $O(n)$ action but commutes with it up to a non-trivial automorphism of $O(n)$.

\begin{proof}
We denote by $X\mapsto X^{\wedge}$ the pro-unipotent completion of a space. This is related to the rationalization in nice cases (see \cite[Proposition 2.11]{boavidaformality}) but better behaved for our purposes in non-simply connected situations.

In order to construct automorphisms of $(\mathcal{D}_n)_{\QQ}$, we use the following sequence of equivalences from \cite{boavidaformality}:
\[\Aut((\Dd_n)_{\QQ})\to\Aut({(\Dd_n^\wedge)}_{\QQ})\to \Aut({({j^*\Dd_n^\wedge})}_{\QQ})\to \Aut({({j^*\Dd_2^\wedge})}_{\QQ}^{\boxtimes \lfloor n/2\rfloor}(\boxtimes {j^*\Dd_1})) \]
where in the last automorphism space the factor $j^*(\Dd_1)$ is present if $n$ is odd and absent if $n$ is even. The first of these maps is a weak equivalence by the fact that $\Dd_n$ is simply connected \cite[Proposition 2.11]{boavidaformality}. For the second map, $j^*$ denotes the restriction functor induced by a certain functor $j:\cat{simp(Fin)}\to\cat{Tree}$ from the category of simplices of the nerve of $\cat{Fin}$. It induces an equivalence by \cite[Proposition 6.3]{boavidaformality}. Finally the last map is an equivalence by the additivity theorem of \cite{boavidaweiss} which remains true after completion by \cite[Theorem 5.5]{boavidaformality}.

Now, the operad $PaB_{\QQ}$ in the previous proposition is a model for $\Dd_2^\wedge$, the pro-unipotent completion of the operad $\mathcal{D}_2$. Let us write $\varphi'$ for the automorphism of $\Dd_2^\wedge$ corresponding to $\mu$ in  Proposition \ref{prop : dimension2}. In the even case, consider the box product of $n/2$-copies of $\varphi'$ and in the odd case, consider the box product of $\lfloor n/2\rfloor$ copies of $\varphi'$ with one copy of the identity of $\Dd_1$. In either case, this produces an automorphism of $(\Dd_n)_\QQ$ thanks to the sequence of equivalences above. We call this automorphism $\varphi$ and check that it satisfies the three conditions. For condition (1), this follows from the fact that this automorphism acts with degree $\mu^{\lfloor n/2\rfloor}$ in arity $2$ \cite[Propositions 7.2 and 8.2]{boavidaformality}.

It remains to prove condition (2) and (3). We first prove that the analogous conditions are satisfied with $O(n)$ replaced by its maximal torus $\TT$ (see subsection \ref{subsection : orthogonal stuff} for details about the maximal torus and its Weyl group). In the odd or even dimensional cases, the decomposition that we have chosen is compatible with the action of $\TT$ so (2) follows easily. Part (3) can be reduced to the case $n=2$ by additivity. But this case follows from the fact that
\[\Aut^h((\Dd_2)_{\QQ})\simeq \GT_{\QQ}\ltimes SO(2)_{\QQ}\]
where the semi-direct product structure is defined through the cyclotomic character (see \cite{Fresse}). In particular our automorphism $\varphi$ acts with degree $\mu$ on $SO(2)_\QQ=S^1_{\QQ}$ which implies (3).

Now we prove (2) and (3) for the group $O(n)$. For this it suffices to use the equation
\[((\Dd_n)_{\QQ})_{hO(n)}\simeq (((\Dd_n)_{\QQ})_{h\TT})_{h\Weyl}.\]
Therefore, we just need to prove that the action of $\varphi$ on $((\Dd_n)_{\QQ})_{h\TT}$ is compatible with the $\Weyl$-action. This is clear for the symmetric group part of $\Weyl$ (see subsection \ref{subsection : orthogonal stuff} for an explicit description of $\Weyl$). For the $(C_2)^{\lfloor n/2\rfloor}$ subgroup, this comes from the fact that the automorphism $\varphi'$ of $(\Dd_2)_{\QQ}$ that we had chosen commutes with $\tau$.
\end{proof}

\section{Prelude}

Before proving the main theorem let us illustrate how the
abstract machinery of Section \ref{darkmagic} on purity and filtered formality may be applied to a couple of simple particular situations.

First, let us consider $\CC P^n$ with its $G:=PGL_{n+1}(\CC)$-action. Then we have
\[E_1^{-i,j}(\Aa_G(\CC P^n))=H^{j-2i}(BG)\otimes H^i(\CC P^n)\Rightarrow H^{j-i}_G(\CC P^n).\]
Since both $\CC P^n$ and $G$ are defined over $\ZZ$, by usual étale cohomology arguments, after extending scalars to $\QQ_\ell$ we can equip $\Aa_G(\CC P^n)$ with a Frobenius action $\varphi$ in such a way that
$E_1^{-i,j}(\Aa_G(\CC P^n))$ becomes pure of weight $j-i$. This gives that, for all $i$, $H_G^i(\CC P^n;\QQ_\ell)$ is pure of weight $i$. By Theorem \ref{purepurity} we obtain that $\CC P^n$ is $G$-equivariantly formal over $\QQ_\ell$. By descent of formality for maps of commutative dg-algebras the same is true over $\QQ$.

\begin{remark}
Note  that even for $n=1$, the Leray-Serre spectral sequence does not collapse multiplicatively. Indeed, in that case, the inclusion $U(1)\to PGL_2(\CC)$ given by
\[a\mapsto \begin{bmatrix}
a & 0\\
0& 1
\end{bmatrix}\]
is a homotopy equivalence and the computation is the same as in Example \ref{eq:2pts}.
\end{remark}

A similar argument gives the following general result.

\begin{theorem}
 Let $X$ be a smooth complex projective variety together with an action of an algebraic group $G$.
 Then $X$ is $G$-equivariantly formal.
\end{theorem}

To prove this, one just needs to enlarge the notion of weight by considering complexes $A$ over $\QQ_\ell[\varphi]$ such that the eigenvalues of $\varphi^*$ on the cohomology $H^i(A)$ are Weil numbers of weight $i$. The splitting Theorem \ref{purepurity} works as well for this bigger category.

\begin{remark}
In \cite{Lilly}, it is shown that compact Kähler manifolds
with holomorphic actions are equivariantly formal under the assumption that the Leray-Serre spectral sequence degenerates at $E_2$. Note that our argument shows that this is always the case for smooth projective varieties.
Also, \cite{ScullKahler} proves formality for holomorphic torus actions on simply connected compact Kähler manifolds, in the genuine equivariant setting.
\end{remark}

As another illustrative example, we consider configuration spaces of points in $\CC^n$, together with their $GL_n(\CC)$-action.
Let us first recall their cohomology, which will also be useful in proving the main theorem.

\begin{theorem}[Arnold, Cohen]\label{cohomconf}
There is an isomorphism of rings
\[H^*(\Conf_k(\RR^n))\cong \QQ[x_{ij}]/(x_{ij}^2, x_{ij}-(-1)^{n}x_{ji},x_{ij}x_{jk}+x_{jk}x_{ki}+x_{ki}x_{ij})\]
where $1\leq i,j\leq k$ and $i \neq j$ and the classes $x_{ij}$ are of degree $n-1$.
\end{theorem}

\begin{theorem}\label{theo : conf is formal}
The space $\Conf_k(\CC^n)$ is $GL_n(\CC)$-equivariantly formal over $\QQ$.
\end{theorem}

\begin{proof}
 To lighten notation, we write $X=\Conf_k(\CC^n)$ and $G=GL_n(\CC)$. The action is algebraic so, using \'etale cohomology arguments, after extending scalars to $\QQ_\ell$ we can equip both $X$ and $BG$ with a Frobenius action.
In this case, $H^{i}(BG;\QQ_\ell)$ is pure of weight $i$ and $H^j(X;\QQ_\ell)$ is pure of weight $2nj/(2n-1)$ (see \cite{CiHo2}).
This gives that, over $\QQ_\ell$, the group
\[E_1^{-i,j}(\Aa_G(X))=H^{j-2i}(BG;\QQ_\ell)\otimes H^i(X;\QQ_\ell)\]
is pure of weight
${{1}\over{2n-1}}\left(j(2n-1)+i(2-2n)\right)$.
Therefore $\Aa_G(X)\otimes\QQ_\ell$ lands in the category $\FCxp{2n-1}{\QQ_\ell[\varphi]}$, with
$\alpha={1\over 2n-1}$.
By Theorem \ref{theoEis1}, we obtain that a model for $\Aa_{\pl}(X_{hG})\otimes\QQ_\ell$ is given by the $E_{2n-1}$-page of $\Aa_G(X)\otimes\QQ_\ell$, which is also the $E_{2n}$-page of the Leray-Serre spectral sequence. The differential of this page is determined by $d_{2n}:H^{2n-1}(X)\to H^{2n}(BG)$ since the cohomology of $X$ is generated as a ring by $H^{2n-1}(X)$. We claim that, for all $i$ and $j$, the map $d_{2n}$ sends the generator $x_{ij}$ to $ c_n$ where $c_n\in H^{2n}(BG)$ is the top Chern class.
\begin{center}

\begin{tikzpicture}
\draw[thick] [->] (0.1,1.75)   -- (2.9,0.1);
\draw[thick]  (0,0)   -- (3.5,0);
\draw[thick]  (0,0)   -- (0,2.5);
\draw (0,1.75)  node[left] {$_{2n-1}$} node{$\bullet$};
\draw (0.5,0)  node[below] {$_2$} node{$\bullet$};
\draw (1,0)  node[below] {$_4$} node{$\bullet$};
\draw (2,0)  node[below] {$\cdots$};
\draw (3,0)  node[below] {$_{2n}$} node{$\bullet$};
\draw (1.8,0.8)  node[above] {$_{d_{2n}}$};
\draw (-1,1) node {${}^{\mathrm{LS}}E_{2n}=$};
\end{tikzpicture}
\end{center}
In order to see this we may use the map $\Conf_k(\CC^n)\to \Conf_2(\CC^n)$ that forgets all the points but $i$ and $j$. This map is $G$-equivariant and functoriality of the spectral sequence shows that it suffices to prove the claim in the case $k=2$. In that case, the Borel fibration for the $GL_n(\CC)$-action on $\CC^n-\{0\}$ is identified with $BGL_{n-1}(\CC) \to BGL_{n}(\CC)$ and it is classical that the transgression $d_{2n}$ maps the generator of $H^{2n-1}(\CC-\{0\})$ to the Euler class, which coincides with the top Chern class.

Now, to prove formality of $\Aa_{\pl}(X_{hG})\otimes\QQ_\ell$ it suffices to build a quasi-isomorphism of commutative dg-algebras between $({}^{\mathrm{LS}}E_{2n},d_{2n})$ and its cohomology.
The cohomology is given by
\[{}^{\mathrm{LS}}E_{2n+1}=(\Ker\,d_{2n}\cap H^*(X))\otimes H^*(BG)/(c_{n}).\]
The quotient $H^*(BG)/(c_{n})$ is isomorphic to a subalgebra of $H^*(BG)$
so we obtain an inclusion \[({}^{\mathrm{LS}}E_{2n+1},0)\hookrightarrow ({}^{\mathrm{LS}}E_{2n},d_{2n})\] which is obviously a quasi-isomorphism. This proves $G$-equivariant formality of $X$ over $\QQ_\ell$, which can then be descended to $\QQ$.
\end{proof}

\begin{remark}
The theorem above recovers the known fact that $\mathcal{M}_{0,k+1}$, the moduli space of genus 0 curves with $k+1$ marked points, is formal. Indeed, the Borel construction $\Conf_k(\CC)_{hGL_1(\CC)}$ is equivalent to the quotient $\Conf_k(\CC)/GL_1(\CC)$, since the action is free, and this quotient is equivalent to $\mathcal{M}_{0,k+1} = \Conf_{k+1}(\mathbb{P}^1)/PGL_2(\CC)$.
\end{remark}

\begin{remark}
The theorems above capture the essence of the proof of our main theorem. However, there are several reasons that made the strategy explained above insufficient for our purposes.
\begin{enumerate}
\item It only gives formality with $\QQ_\ell$-coefficients.
\item It is not clear that the Galois action will be compatible with operadic composition.
\item It only works in even dimensions.
\item In the even dimensional case, it gives $GL_n(\CC)\simeq U(n)$-equivariant formality instead of $O(2n)$-formality.
\end{enumerate}
We believe that the first three problems could be overcome. The first one is the very classical problem of descent of formality which is well-known to hold for individual commutative dg-algebras (see \cite{HS}) and may be easily extended to morphisms, but a proof of functorial descent of formality seems to be missing from the literature. To solve the second one, we would need to argue that the little disks operad is an algebro-geometric object so that the Galois action is compatible with operadic composition. This is exactly the content of the paper \cite{BHK} (see also \cite{vaintrobmoduli} in the case of the two-dimensional little disks operad). The third problem could be solved by additivity as in the present paper. The last problem, though, seems very difficult to avoid as it is a classical fact that the only element of the absolute Galois group of $\QQ$ that commutes with complex conjugation is the identity. So, there does not seem to be possible to construct an automorphism of $\Dd_n$ that satisfies the requirements of Proposition \ref{automorphism} using merely the absolute Galois group action.
\end{remark}

\section{Proof of the theorem}

Let $G$ be $O(n)$ or $SO(n)$. From the previous sections
we have a functor $T\mapsto \Aa_G(\Dd_n(T))$ from the category $\cat{Tree}$ to filtered commutative dg-algebras over $\QQ$. This has an associated spectral sequence and we
consider the functor given by its $E_{n-1}$-page with its associated differential (see \ref{prop functoriality}). If $n$ is odd, then this functor has trivial differential. The following proposition deals with the case when $n$ is even.

\begin{proposition}\label{penultimate}
The functor $T\mapsto (E_{2n-1}(\Aa_G(\Dd_{2n}(T))),d_{2n-1})$, from $\cat{Tree}$ to $\cat{CDGA}_{\QQ}$, is formal, for
$G$ equal to $SO(2n)$ and $O(2n)$.
\end{proposition}

\begin{proof}
We first consider the case when $G=SO(2n)$.
The category $\cat{Tree}$ contains two types of objects: the linear trees and the non-linear trees. If $T$ is linear, $\Dd_{2n}(T)\simeq *$ and the Leray-Serre spectral sequence is simply
\[{}^{\mathrm{LS}}E_2^{i,j}=H^i(BSO(2n))\Rightarrow H^i(BSO(2n))\]
In this case, formality is trivial. If $T$ is non-linear, the Leray-Serre spectral sequence is given by
\[{}^{\mathrm{LS}}E_2^{i,j}=H^i(BSO(2n))\otimes H^j(\Dd_{2n}(T))\Rightarrow H^{i+j}_{SO(2n)}(\Dd_{2n}(T))\]
By Proposition \ref{cohomconf} the cohomology algebra $H^*(\Dd_{2n}(T))$ is generated by $H^{2n-1}(\Dd_{2n}(T))$ so by multiplicativity of the spectral sequence, the differential $d_{2n}$ is determined by its restriction to $H^{2n-1}(\Dd_{2n}(T))$. This restriction sends all the algebra generators to the Euler class $e$ in $H^{2n}(BSO(2n))$. This can be seen as in the proof of Theorem \ref{theo : conf is formal} replacing $BGL_n(\CC)$ with $BSO(2n)$. The quotient $H^*(BSO(2n))/(e)$ is isomorphic to $H^*(BSO(2n-1))$ and can naturally be viewed as a subalgebra of $H^*(BSO(2n))$, the subalgebra generated by Pontryagin classes $p_i$ for $i<n$. It follows that we have an  injection
\begin{equation}\tag{$\divideontimes$}\label{inclusion}
H^*(BSO(2n-1))\otimes (\Ker\,d_{2n}\cap H^*(\Dd_{2n}(T)))\hookrightarrow {}^{\mathrm{LS}}E_{2n}=H^*(BSO(2n))\otimes H^*(\Dd_{2n}(T))\end{equation}
which is a map of commutative dg-algebras. Moreover, it is straightforward to check that this map is a quasi-isomorphism.

Let us denote by $T\mapsto H(T)$ the assignment given by sending $T$ to $H^*(BSO(2n))$ if $T$ is linear or \[H^*(BSO(2n-1))\otimes (\Ker\,d_{2n}\cap H^*(\Dd_{2n}(T)))\] if $T$ is not linear. We claim that this assembles into a functor on $\cat{Tree}$. The functoriality is obvious for maps between linear trees and for maps between non-linear trees. Now, we observe that in $\cat{Tree}$, there are no maps from non-linear trees to linear trees. So it simply remains to define the functoriality for maps from a linear tree to a non-linear tree. In that case, we have the map
\[H^*(BSO(2n))\cong H^*(BSO(2n))\otimes \QQ \xrightarrow{\rho\otimes 1}H^*(BSO(2n-1))\otimes (\Ker\,d_{2n}\cap H^*(\Dd_{2n}(T)))\]
where $\rho$ is the canonical map $H^*(BSO(2n))\to H^*(BSO(2n-1))$.

We have just explained how to construct, for each $T$, a quasi-isomorphism \[H(T)\to E_{2n-1}(\Aa_{SO(2n)}(\Dd_{2n}(T))).\] One easily checks that this is a natural quasi-isomorphism.
This proves formality of the commutative dg-algebra $E_{2n-1}(\Aa_{SO(2n)}(\Dd_{2n}(T)))$.

To prove formality for the algebra $E_{2n-1}(\Aa_{O(2n)}(\Dd_{2n}(T)))$, we use that there is a $C_2$-action on $E_{2n-1}(\Aa_{SO(2n)}(\Dd_{2n}(T)))$ whose fixed points is precisely $E_{2n-1}(\Aa_{O(2n)}(\Dd_{2n}(T)))$. We construct a $C_2$-action on $H(T)$ using the obvious action on $H^*(BSO(2n))$ if $T$ is linear and the restriction of the $C_2$-action of $E_{2n-1}(\Aa_{SO(2n)}(\Dd_{2n}(T)))$
along the inclusion (\ref{inclusion}) when $T$ is not linear. This $C_2$-action is functorial for trees and the map $H(T)\to E_{2n-1}(\Aa_{SO(2n)}(\Dd_{2n}(T)))$ defined above is $C_2$-equivariant. Therefore, taking $C_2$-fixed points of this map gives functorial formality of $E_{2n-1}(\Aa_{O(2n)}(\Dd_{2n}(T)))$.
\end{proof}

\begin{theorem}\label{maintheo}Let $G$ be $O(n)$ or $SO(n)$.
The little $n$-disks operad $\Dd_n$ is $G$-equivariantly formal.
\end{theorem}
\begin{proof}
Let $\xi$ be a square in $\QQ^{\times}$ different form $\pm 1$. Let $\mu$ be a square root of $\xi$. Using this choice of $\mu$ in Proposition \ref{automorphism}, we can promote the functor $\Aa_{G}(\Dd_n(T))$ to a functor from $\cat{Tree}$ to the category $\FCx{0}{\QQ[\varphi]}$  of filtered complexes of $\QQ[\varphi]$-modules up to $E_1$-isomorphisms. Composed with the forgetful functor
to complexes of $\QQ$-vector spaces  we recover the functor
 \[T\mapsto \Aa_{\pl}(\Dd_n(T)_{hG}).\]

In the odd case, we get that $E_1^{-i,j}(\Aa_G(\Dd_n))$ is pure of weight $j-i$, so Theorem \ref{purepurity} implies formality.

In the even case, we get that $E_1^{-i,j}(\Aa_G(\Dd_n))$ is pure of weight $\frac{1}{n-1}\left[i(n-2)+j(n-1)\right]$ so this fits in Theorem \ref{theoEis1} with $\alpha=1/(n-1)$ and $r=n-1$. Therefore we get a functorial model of  $\Aa_{\pl}(\Dd_n(T)_{hG})$, given by the
penultimate page of the Leray-Serre spectral sequence, together with its differential. The result now follows from Proposition \ref{penultimate}.
\end{proof}

For applications in manifold calculus, it is useful to understand the $O(n)$-equivariant homotopy type of the map $\Dd_n\to\Dd_{n+c}$ induced by the standard embedding $\RR^n\to\RR^{n+c}$. In this case, we have the following variant of our main theorem.

\begin{theorem}
Let $n\geq 3$ and $c\geq 2$. Let $G=O(n)$ or $G=SO(n)$. The map $\Dd_n\to\Dd_{n+c}$ induced by the standard inclusion of $\RR^n$ in $\RR^{n+c}$ is $G$-equivariantly formal.
\end{theorem}

\begin{proof}
As in Proposition \ref{automorphism}, for any two units $\mu$ and $\nu$ in $\QQ$, we may construct an automorphism $\varphi$ of the map $(\Dd_n)_\QQ\to(\Dd_{n+c})_\QQ$ satisfying the following conditions
\begin{enumerate}
\item The induced action of $\varphi^*$ on $H^{i}(\Dd_n)$ is by multiplication by $\mu^{i/2}$ if $n$ is odd and by $\mu^{ni/2(n-1)}$ if $n$ is even.
\item The induced action on $H^{i}(\Dd_{n+c})$ is by multiplication by $(\mu^{(n-1)/2}\nu)^{i/(n+c-1)}$ if $n$ is odd or $(\mu^{n/2}\nu)^{i/(n+c-1)}$ if $n$ is even.
\item The automorphism $\varphi$ extends to an automorphism of the map of $\infty$-operads
\[((\Dd_n)_{\QQ})_{hO(n)}\to ((\Dd_{n+c})_{\QQ})_{hO(n)}.\]
\item The induced action of $\varphi^*$ on $H^{i}(BO(n))$ is by multiplication by $\mu^{i/2}$.
\end{enumerate}
This is constructed by taking the box product of the automorphism that we construct in Proposition \ref{automorphism} and any automorphism of $(\Dd_c)_{\QQ}$ that acts with degree $\nu$ in arity $2$ (which exists by \cite[Theorem 7.1]{boavidaformality}).

Now, take $\xi=\sqrt{\mu}$ and $\nu=\xi^{c}$ if $n$ is odd or $\nu=\xi^{\frac{cn}{n-1}}$ if $n$ is even (since we have freedom in $\mu$, we can assume that these powers exist in $\QQ^{\times}$). Then, in the even case, the map
\[\Aa_G(\Dd_{n+c})\to\Aa_G(\Dd_n)\]
gets promoted to a map in $\FCxp{r}{\QQ[\varphi]}$ with $\alpha=1/(n-1)$ and $r=n-1$. In the odd case, it is a map in the category of cochain complexes satisfying the condition of Theorem \ref{purepurity}. In either case, we get formality of the map.
\end{proof}

\section{Equivariant cohomology ring of configuration spaces}\label{compute}
Let $G$ be $SO(m)$ or $O(m)$.
By Theorem \ref{maintheo}, a model for the Borel construction of $G$ acting on $\Conf_\ell(\RR^{m})$
is given by the rational equivariant cohomology ring. In this section, we compute these cohomology rings. The cases $m$ even or $m$ odd behave very differently.

\subsection{Examples in the case of two points}
Before we explain the general case, we illustrate this in two examples, for two points and in low dimension $m$.

\begin{example}\label{eq:2pts}
Consider the action of $S^1$ on $\Conf_2(\RR^3)$.
The $E_2$-page splits as a tensor product $H^*(BS^1)\otimes H^*(S^2)$. For degree reasons there are no differentials in this spectral sequence. But the $E_2$ page is not multiplicatively isomorphic to the cohomology of the Borel construction, i.e. the Borel fibration is not rationally trivial.

To describe the equivariant cohomology \emph{ring}, we observe that the action of $S^1$ on $\Conf_2(\RR^3) \simeq S^2$ is the suspension of the action of $S^1$ on itself, and so the Borel construction is the wedge $BS^1 \vee BS^1$. This gives the cohomology ring
\[
H^*(\Conf_2(\RR^3)_{hS^1}) \cong \QQ[a,b]/(ab)
\]
where $|a| = |b| = 2$. The action of $\Weyl=C_2$ on $BS^1 \vee BS^1$ exchanges the summands, so exchanges $a$ and $b$. (It corresponds to the degree one map on $S^2$ given by $(x,y,z) \mapsto (y,x,-z)$.) The homotopy fiber sequence \[
\Conf_2(\RR^3) \to \Conf_2(\RR^3)_{hS^1} \to BS^1
\]
is equivalent to the homotopy fiber sequence
$S^2 \to \CC P^\infty \vee \CC P^\infty \to \CC P^\infty$,
where the first map is the fold map $S^2 \to S^2 \vee S^2$ composed with the inclusion of the $2$-skeleton. In cohomology, this sequence has the form $\QQ[q] \to \QQ[a,b]/(ab) \to \QQ[x]/x^2$ where $q \mapsto a-b$ and $a, b\mapsto x$. Changing variables, $x = a$ and $q = a-b$, we have
\[H^*(\Conf_2(\RR^3)_{hS^1})\cong \QQ[x,q]/x(x-q)\,,\]
where $|x|=|q|=2$. The corresponding $C_2$-action on this algebra sends $x \mapsto x-q$ and $q \mapsto -q$. Setting $y = 2x-q$, we have $H^*(\Conf_2(\RR^3)_{hS^1}) \cong \QQ[y,q]/(y^2 - q^2)$. Now,
\[
H^*(\Conf_2(\RR^3)_{hSO(3)}) \cong H^*(\Conf_2(\RR^3)_{hS^1})^{C_2}
\]
and it follows that $H^*(\Conf_2(\RR^3)_{hSO(3)}) \cong \QQ[y, q^2]/(y^2 - q^2)$. This calculation is consistent with the fact that the $SO(3)$-homotopy orbits of $S^2$ is equivalent to $BS^1$.

As for the $\Sigma_2$-action on $\Conf_2(\RR^3)_{hS^1}$ which permutes labels, it sends $x \mapsto q - x$ and $q \mapsto q$. And for $\Conf_2(\RR^3)_{hSO(3)}$ it sends $y$ to $-y$.
\end{example}

\begin{example}
Now, consider the case of $\Conf_2(\RR^4)\simeq S^3$ with its $O(4)$-action. First, we restrict this action to the maximal torus $\TT=S^1\times S^1$. The Borel construction is then given by the wedge $BS^1\vee BS^1$ as can be seen from the following $\TT$-equivariant homotopy cocartesian square
\[
\xymatrix{
S^1\times S^1\ar[d]\ar[r]&S^1\times D^2\ar[d]\\
D^2\times S^1\ar[r]&S^3
}
\]
In this case the $E_2$ page of the Leray-Serre spectral sequence is given by
\[E_2^{*,*}=H^*(B\TT)\otimes H^*(S^3)=\QQ[q_1,q_2]\otimes\QQ[x]/x^2\]
There is a unique differential $d_4$ determined by
\[d_4(x)=q_1q_2\]
The $E_5$-page is thus $\QQ[q_1,q_2]/(q_1q_2)$ which is indeed the cohomology of the wedge.

Finally we consider the action of $SO(4)$ and $O(4)$. The spectral sequences are the Weyl fixed points of the previous one. Here $\Weyl$ is $\Sigma_2\ltimes C_2$ and $\Sigma_2\ltimes (C_2)^2$, respectively. For $SO(4)$, the Weyl fixed points of the $E_2$-page can be identified with the subalgebra generated by $q_1^2+q_2^2$, $q_1q_2$ and $x$ with differential $d_4(x)=q_1q_2$. Therefore the $E_5$-page is
\[E_5^{*,*}=\QQ[q_1^2+q_2^2].\]
For $O(4)$, the Weyl fixed points of the $E_2$-page
is
$\QQ[q_1^2+q_2^2]$ and so in this case $d_4\equiv 0$.
 In fact, the Borel construction relative to $O(4)$ and $SO(4)$ is in both cases rationally equivalent to $\mathbb{H} P^\infty$.
\end{example}

\subsection{General case in even dimensions}
In the even dimensional case $m=2n$, the Leray-Serre spectral sequence does not degenerate at the $E_2$-page, but the $E_{2n+1}$-page is multiplicatively isomorphic to the cohomology. For $G=SO(2n)$, we have
\[H^*_{SO(2n)}(\Conf_\ell(\RR^{2n}))\cong
H^*(BSO(2n-1))\otimes K,
\]
where $K$ is the kernel of the differential $d_{2n}$ restricted to the zeroth-column of the Leray-Serre spectral sequence, i.e. $H^{*}(\Conf_{\ell}(\RR^{2n}))$. It is determined by $d_{2n}(x_{ij})=e$, where $e$ is the Euler class.

For $G=O(2n)$ we have to take $C_2$-fixed points of the above formula. This gives
\[H^*_{O(2n)}(\Conf_\ell(\RR^{2n}))\cong
H^*(BSO(2n-1))\otimes K^{C_2}.
\]
The $C_2$-fixed points of $H^{*}(\Conf_{\ell}(\RR^{2n}))$ are the elements of degree $2i(2n-1)$, $i \geq 1$. So $K^{C_2}$ is simply the intersection of $K$ with $\oplus_{i} H^{2i(2n-1)}(\Conf_{\ell}(\RR^{2n}))$.

\subsection{General case in odd dimensions}
In the odd dimensional case, the situation is  harder. The Leray-Serre spectral sequence degenerates at the $E_2$-page, but not multiplicatively. So we must argue differently. Although we take coefficients in $\QQ$, the results in this section should apply if coefficients are taken in a ring where $2$ is invertible.

\begin{proposition}\label{prop:Tncohomology}
The $\TT^n$-equivariant cohomology of $\Conf_\ell(\RR^{2n+1})$ is isomorphic to the algebra generated by $y_{ij}$, ${1 \leq i < j \leq \ell}$, and $q_1, \dots, q_n$, with $y_{ij}$ of degree $2n$ and $q_u$ of degree $2$, modulo the modified Arnol'd relation
\[
y_{ij} y_{jk} - y_{jk}y_{ik} - y_{ik}y_{ij} + p_n = 0
\]
and the relation $y_{ij}^2 = p_n$\,, where $p_n = (q_1 \dots q_n)^2$.
\end{proposition}
\begin{proof}
We make a start by setting $\ell = 2$. The action of $\TT^n$ on $\Conf_2(\RR^{2n+1}) \simeq S^{2n}$ is the suspension of the action of $\TT^n$ on the equator $S^{2n-1}$. Therefore, the homotopy orbits of the $\TT^n$-action on $S^{2n}$ is identified with the pushout of
\[
B \TT^n \gets (S^{2n-1})_{h \TT^n} \to B \TT^n
\]
The cohomology of this pushout is the fiber product
\[
P := H^*(B \TT^n) \times_{H^*(S^{2n-1}_{h \TT^n})} H^*(B \TT^n) \;.
\]
We know $H^*(S^{2n-1}_{h \TT^n}) \cong \QQ[q_1, \dots, q_n]/(q_1 \dots q_n)$ (e.g. from proposition \ref{penultimate} or by noticing that $S^{2n-1}_{h \TT^n}$ is the base change of $BSO(2n-1) \to BSO(2n)$ along the inclusion $B\TT^n \to BSO(2n)$ and the Euler class $e \in H^{2n}(BSO(2n))$ is mapped to $q_1\dots q_n$). And the two maps in the fiber product
\[
\QQ[q_1, \dots, q_n] \to \QQ[q_1, \dots, q_n]/(q_1 \dots q_n)
\]
send $q_i$ to $q_i$ and $-q_i$, respectively. (The difference in sign has to do with the fact that the action of $\TT^n$ on the top hemisphere of $S^{2n}$ is the reflection of the action on the bottom hemisphere.) An element of $P$ corresponds to a pair $(f^{-},f + q_1\dots q_n g)$ where $f$, $f^-$ and $g$ belong to $\QQ[q_1,\dots, q_n]$ and $f^-(q_1, \dots, q_n) := f(-q_1, \dots, -q_n)$. Then we may define a map
\[
P \to \QQ[x, q_1, \dots, q_n]/(x^2 - x q_1 \dots q_n)
\]
by sending $(f^-, f + q_1\dots q_n g)$ to $f + xg$, where $x$ is of degree $2n$. This is an isomorphism of $\QQ[q_1,\dots, q_n]$-algebras with inverse sending $q_i$ to $(-q_i, q_i)$ and $x$ to $(0, q_1\dots q_n)$. Under the substitution $y:=2x - q_1 \dots q_n$, we get the presentation $\QQ[y, q_1, \dots, q_n]/(y^2 - p_n^2)$ in the statement. This settles the case of two points.

Moving on, for an arbitrary $\ell$ and $i, j \in \{1, \dots \ell\}$, let $\pi_{ij}$ denote the projection
\[
\Conf_\ell(\RR^{2n+1}) \to \Conf_2(\RR^{2n+1}) \;
\]
and consider the diagram with horizontal maps induced by the projections
\[
\xymatrix{
\otimes_{i < j} H^*(\Conf_2(\RR^{2n+1})_{h \TT^n}) \ar[d]\ar[r]& H^*(\Conf_\ell(\RR^{2n+1})_{h\TT^n})\ar[d]\\
\otimes_{i < j} H^*(\Conf_2(\RR^{2n+1}))\ar[r]&H^*(\Conf_\ell(\RR^{2n+1}))
}
\]
The lower horizontal map is surjective (it enforces the Arnol'd relation). The left-hand map is identified with the tensor product, over $i,j$, of the map
\[
\QQ[x_{ij}, q_1, \dots, q_n]/(x_{ij}^2 - x_{ij} q_1 \dots q_n) \to \QQ[x_{ij}]/x_{ij}^2 \; .
\]
which sends $x_{ij}$ to $x_{ij}$ and each $q_u$ to zero (c.f. example \ref{eq:2pts}). In terms of the generators $y_{ij}:=2 x_{ij} - q_1\dots q_n$, this is the map
\[
\QQ[y_{ij}, q_1, \dots, q_n]/(y_{ij}^2 - p_n^2) \to \QQ[x_{ij}]/x_{ij}^2
\]
which sends $y_{ij}$ to $2 x_{ij}$ and each $q_u$ to zero. Now pick a basis of $H^*(\Conf_\ell(\RR^{2n+1}))$ by monomials in the $2x_{ij}$, call it $\{b_u\}$. This basis lift to a preferred basis $\{b^\prime_u\}$ of $H^*(\Conf_\ell(\RR^{2n+1})_{h \TT^n})$ by lifting it across the diagram above. The map
\[
H^*(\Conf_\ell(\RR^{2n+1})) \otimes H^*(B \TT^n) \to H^*(\Conf_\ell(\RR^{2n+1})_{h \TT^n})
\]
which sends $b_u \otimes q$ to $b_u^\prime q$ is an isomorphism of $H^*(B \TT^n)$-modules by the Leray-Hirsch theorem. (In particular, $2x_{ij}$ is sent to $y_{ij}$ under this isomorphism.) The map
\[
\QQ[\{y_{ij}\}, q_1, \dots, q_n] \to H^*(\Conf_\ell(\RR^{2n+1})_{h\TT^n})
\]
induced by the projection maps $\pi_{ij}$, respects the relation $y_{ij}^2 - p_n^2 = 0$ by construction. As for the Arnol'd relation, it is enough to verify it for three points since $y_{ij} y_{jk} - y_{jk}y_{ik} - y_{ik}y_{ij} + p_n$ is in the image of the map
\[
H^*(\Conf_3(\RR^{2n+1})_{h\TT^n}) \to H^*(\Conf_\ell(\RR^{2n+1})_{h\TT^n})
\]
induced by forgetting all points but $i,j,k$. Now, the composition with
\[
H^*(\Conf_3(\RR^{2n+1})_{h\TT^n}) \to H^*(\Conf_3(\RR^{2n+1}))
\]
satisfies the Arnol'd relation. This means we have relations of the form
\[
(y_{ij}y_{jk} - y_{jk}y_{ik} - y_{ik}y_{ij}) + f_{ij} y_{ij} + f_{jk} y_{jk} + f_{ik}y_{ik} + g(q_1, \dots, q_n) = 0
\]
in $H^*(\Conf_3(\RR^{2n+1})_{h\TT^n})$ for some polynomials $f_{ij}, f_{jk}, f_{ik}$ and $g$ in the variables $q_1, \dots, q_n$. The claim is the polynomials $f$ are zero and $g = p_n$. To see this, take a section $s$ of $p_{ij}$ which adds a far away point in $(0, \infty) \times \{0\} \subset \RR^1 \times \RR^{2n} = \RR^{2n+1}$. By construction, $s$ is $\TT^n$-equivariant and so descends to a map
\[
H^*(\Conf_3(\RR^{2n+1})_{h\TT^n}) \to H^*(\Conf_2(\RR^{2n+1})_{h\TT^n})
\]
sending $y_{ij}$ to $y_{ij}$, and sending $y_{jk}$ and $y_{ik}$ to $q_1\dots q_n$. The image of our relation reads
\[
y_{ij} q - q^2 - q y_{ij} + f_{ij} y_{ij} + f_{jk}q + f_{ij} q + g = 0
\]
where $q$ is short for $q_1 \dots q_n$. As such, $f_{ij}$ must be zero. In the same way, by choosing an appropriate section of $p_{jk}$ and $p_{ik}$, we conclude that $f_{jk}$ and $f_{ik}$ are also zero. And then we are left with the identity $g = q^2$.

Having described the map
\begin{equation*}\label{maprel}
\QQ[\{y_{ij}\}, q_1, \dots, q_n] / \mbox{relations} \to H^*(\Conf_\ell(\RR^{2n+1})_{h\TT^n})
\end{equation*}
what remains to do is showing that it is an isomorphism. By dimension counting, e.g. inducting up on length of the monomials in the $y_{ij}$, the two vector spaces have the same dimension. Finally, the map is surjective since every basis element $b_u^\prime$ is in the image.
\end{proof}

\begin{remark}
Here is an illuminating graphical description of the algebra from proposition \ref{prop:Tncohomology}. The elements are linear combinations of graphs with set of vertices $\{1, \dots, \ell \}$ with coefficients in $\QQ[q_1, \dots, q_n]$. These graphs are allowed to have multiple edges connecting vertices but no edges connecting a vertex to itself. However, the edges are not ordered, nor oriented. The relations on these graphs are of two kinds: the relation which says that a double edge may be removed at the expense of multiplying the resulting graph by $p_n = (q_1 \dots q_n)^2$, and the modified Arnol'd relation. Below is the graphical representation of the relations.
\smallskip
\begin{center}
\begin{tikzpicture}
\node (A) [] at (0,0) {};
\node (B) [] at (2,0) {};
\node (C) [] at (5,0) {};
\node (D) [] at (7,0) {};

\draw (A)  node[below] {$_i$} node{$\bullet$};
\draw (B)  node[below] {$_j$} node{$\bullet$};
\path [bend left=40] (A) edge (B);
\path [bend right=40] (A) edge (B);

\draw (3,0)   node{$=$};
\draw (4.2,0)   node{$p_n$};
\draw (C)  node[below] {$_i$} node{$\bullet$};
\draw (D)  node[below] {$_j$} node{$\bullet$};
\end{tikzpicture}

\begin{tikzpicture}
\node (A) [] at (0,0) {};
\node (B) [] at (1,0) {};
\node (C) [] at (2,0) {};
\node (D) [] at (3,0) {};
\node (E) [] at (4,0) {};
\node (F) [] at (5,0) {};
\node (G) [] at (6,0) {};
\node (H) [] at (7,0) {};
\node (I) [] at (8,0) {};
\node (J) [] at (10,0) {};
\node (K) [] at (11,0) {};
\node (L) [] at (12,0) {};

\draw (A)  node[below] {$_i$} node{$\bullet$};
\draw (B)  node[below] {$_j$} node{$\bullet$};
\draw (C)  node[below] {$_k$} node{$\bullet$};
\path [bend left=40] (A) edge (B);
\path [bend left=40] (B) edge (C);
\draw (2.5,0)   node{$-$};
\path [bend left=40] (D) edge (F);
\path [bend left=40] (D) edge (E);
\draw (E)  node[below] {$_j$} node{$\bullet$};
\draw (D)  node[below] {$_i$} node{$\bullet$};
\draw (F)  node[below] {$_k$} node{$\bullet$};
\draw (5.5,0)   node{$-$};
\path [bend left=40] (G) edge (I);
\path [bend left=40] (H) edge (I);
\draw (G)  node[below] {$_i$} node{$\bullet$};
\draw (H)  node[below] {$_j$} node{$\bullet$};
\draw (I)  node[below] {$_k$} node{$\bullet$};
\draw (8.7,0)   node{$\quad = \; - p_n$};
\draw (J)  node[below] {$_i$} node{$\bullet$};
\draw (K)  node[below] {$_j$} node{$\bullet$};
\draw (L)  node[below] {$_k$} node{$\bullet$};
\end{tikzpicture}
\end{center}
The multiplication is given by superposition (or union) of graphs and multiplication of coefficients.
The non-equivariant cohomology of $\Conf_\ell(\RR^{2n+1})$ (Theorem \ref{cohomconf}) has a similar, and well-known, graphical description. In cohomology, the map $\Conf_\ell(\RR^{2n+1}) \to \Conf_\ell(\RR^{2n+1})_{h\TT^n}$ is given by

\smallskip
\begin{center}
\begin{tikzpicture}
\node (A) [] at (0,0) {};
\node (B) [] at (2,0) {};
\node (C) [] at (5,0) {};
\node (D) [] at (7,0) {};

\draw (A)  node[below] {$_i$} node{$\bullet$};
\draw (B)  node[below] {$_j$} node{$\bullet$};
\path (A) edge (B);
\path (C) edge (D);

\draw (3,0)   node{$\mapsto$};
\draw (4.2,0)   node{$2$};
\draw (C)  node[below] {$_i$} node{$\bullet$};
\draw (D)  node[below] {$_j$} node{$\bullet$};
\end{tikzpicture}
\end{center}
\end{remark}

The $SO(2n+1)$-equivariant cohomology of $\Conf_\ell(\RR^{2n+1})$ is isomorphic to the $\Weyl$-fixed points of the $\TT^n$-equivariant cohomology, where $\Weyl \subset (\Sigma_n\ltimes C_2^n)\times C_2$ is the Weyl group as in section \ref{subsection : orthogonal stuff}. The Weyl group acts on the coefficients $\QQ[q_1, \dots, q_n]$ by permuting the $q_i$ and changing sign. And it fixes the generators $y_{ij}$. (To see this, recall from the proof of proposition \ref{prop:Tncohomology} that $y_{ij}=2x_{ij} - q_1 \dots q_n$. An element $(\sigma, \epsilon_1, \dots, \epsilon_n, \eta) \in \Weyl$ acts as $q_1 \dots q_n \mapsto \eta q_1 \dots q_n$ (recall $\eta = \textup{sign}(\sigma) \epsilon_1, \dots, \epsilon_n$), acts as $x_{ij} \mapsto x_{ij} - q_1 \dots q_n$ if $\eta = -1$ and as $x_{ij} \mapsto x_{ij}$ if $\eta = 1$). So we deduce:

\begin{proposition}\label{prop:SOncohomology}
The $SO(2n+1)$-equivariant cohomology of $\Conf_\ell(\RR^{2n+1})$ is isomorphic to the algebra generated by $y_{ij}$, ${1 \leq i < j \leq \ell}$ and $p_1, \dots, p_n$, with $y_{ij}$ of degree $2n$ and $p_u$ of degree $4u$, modulo the Arnol'd relation
\[
y_{ij} y_{jk} - y_{jk}y_{ik} - y_{ik}y_{ij} + p_n = 0
\]
and the relations $y_{ij}^2 = p_n$.
\end{proposition}

Finally, the $O(2n+1)$-equivariant cohomology is the subalgebra generated by $y_{ij} y_{k \ell}$ with $i<j$ and $k < \ell$ and the Pontryagin classes $p_u$ with $u\leq n$. As a vector space is it isomorphic to
\[
H^*(BO(2n+1)) \otimes \left( \oplus_{j} H^{4jn}(\Conf_\ell(\RR^{2n+1})) \right)  \; .
\]

\begin{remark}
The action of a group $G$ on a space $X$ can be encoded in at least three equivalent ways. First, as a map $BG \to B\textup{Aut}^h(X)$. Second, in the Borel fibration $X_{hG} \to BG$. Third, as a map
$X \times G \to X$ (or, rather, a certain simplicial space).
The same is true if we replace \emph{space} by \emph{operad} and also if we rationalize. Khoroskin-Willwacher  \cite{KW} take the first perspective. Using a graph complex model for $B\textup{Aut}^h(\mathcal{D}_n)$ they describe the said map by providing a Maurer-Cartan element in a certain Lie algebra. In this paper, we take the second perspective. So one may wonder how to relate that with the third point of view.

The action map $X \times G \to X$ can be recovered from the Borel fibration as the induced map on homotopy fibers, over $BG$, of the map $X \to X_{hG}$. For the little disks operad with its $SO(n)$ action, our results suggest that, rationally, the action map  factors through $X \times K(\QQ, 4n-1)$ in the odd case and $X \times K(\QQ, n-1)$ in the even case, corresponding to the (shifted) top Pontryagin class and Euler class, respectively. This resonates with the results of Khoroskin-Willwacher.
%
%
%
\end{remark}

\bibliographystyle{amsalpha}
\bibliography{biblio}

\end{document}